\documentclass[letterpaper, 10 pt, journal, twoside]{ieeetran} 
\usepackage{times}
\usepackage{amsthm}
\usepackage{mathrsfs}
\usepackage{amsmath} 
\usepackage{mathtools} 
\usepackage{graphicx,amsmath,amssymb,algorithmic}
\usepackage{cite}
\usepackage{dsfont}
\usepackage{times}
\usepackage{mathrsfs}
\usepackage{psfrag}
\usepackage[colorlinks=true, urlcolor=blue, citecolor=blue, linkcolor=blue]{hyperref}
\usepackage{mathtools} 
\usepackage{graphicx,amsmath,amssymb,algorithmic,color}
\usepackage{cite}
\usepackage{tikz} 
\usepackage{float}
\usepackage{epstopdf}
\usepackage{dsfont}
\newtheorem{thm}{Theorem}
\newtheorem{lem}{Lemma}

\newtheorem{assumption}{Assumption}
\newtheorem{proposition}{Proposition}
\newtheorem{remark}{Remark}
\newcommand{\source}{{THIS IS A PREPRINT VERSION. IF YOU FOUND THIS READING USEFUL FOR YOUR RESEARCH PLEASE CITE THE PUBLISHED VERSION DOI: \href{https://doi.org/10.1109/LCSYS.2021.3136097}{https://doi.org/10.1109/LCSYS.2021.3136097}}}
 \usepackage{fancyhdr}
\makeatletter

\def\ps@IEEEtitlepagestyle{}
\title{\LARGE \bf 
Leader-follower synchronization of a network of boundary-controlled parabolic equations with in-domain coupling}

\author{A. Kabalan, F. Ferrante, G. Casadei, A. Cristofaro, and C. Prieur
\thanks{Abbas Kabalan is with Mines Paris - PSL university, 75006 Paris, France, email: abbas.kabalan@minesparis.psl.eu}
\thanks{Francesco Ferrante is with Department of Engineering, University of Perugia, Via G. Duranti, 67, 06125 Perugia, Italy, email: francesco.ferrante@unipg.it}
\thanks{Giacomo Casadei is with Laboratoire Ampere Dpt. EEA of the \'Ecole Centrale de Lyon, Universit\'e de Lyon, 69134 Ecully, France.}
\thanks{Andrea Cristofaro is with Department of Computer, Control and Management Engineering, Sapienza University of Rome, Italy.}
\thanks{Christophe Prieur is with Univ. Grenoble Alpes, CNRS, Grenoble INP, GIPSA-lab 38000 Grenoble, France.}
\thanks{This work has been partially supported by MIAI@Grenoble Alpes (ANR- 19-P3IA-0003)}
}
\begin{document}        
 \maketitle                                       
\begin{abstract}            
In this paper, we study the leader-synchronization problem for a class of partial differential equations with boundary control and \textit{in-domain} coupling. We describe the problem in an abstract formulation and we specialize it to a network of parabolic partial differential equations. We consider a setting in which a subset of the followers is connected to the leader through a boundary control, while interconnections among the followers are enforced by distributed in-domain couplings. Sufficient conditions in the form of matrix inequalities for the selection of the control parameters enforcing exponential synchronization are given. Numerical simulations illustrate and corroborate the theoretical findings.
\end{abstract}

\section{Introduction}
\subsection{Background and contributions}
\IEEEPARstart{T}{the} problem of consensus and synchronization of multiple agents interacting over a network has been an active domain of research in the past years due to many important applications \cite{olfati2004consensus,dorfler2013synchronization}. Several efforts have been made to develop the theory of synchronization for finite dimensional systems both in the linear \cite{Scardovi2008SynchronizationIN} and the nonlinear case \cite{4287131}.

Recently, researchers have started considering the case in which the agents in the network are infinite-dimensional systems, e.g., systems modeled via partial differential equations (\emph{PDE}s). For these systems, a challenge comes from the fact that sensing and actuation typically take place at the boundary of the domain. First results on synchronization of systems modeled via PDEs can be found in \cite{article}, in which the author considered system modeled by PDEs with in-domain control and in \cite{9120281} with the focus on boundary control. More recently, authors have started considering synchronization with boundary control for different types of PDEs, as in \cite{pilloni2015consensus} for parabolic PDEs and in \cite{aguilar2020leader,chen2020bipartite} for wave equations. A first attempt to consider synchronization for a class of boundary-actuated semilinear PDEs has been proposed in \cite{ferrante2021synchronization}, where the authors considered incremental nonlinearities.

In this paper, we consider parabolic PDEs interacting over a network. The interest behind parabolic PDEs stem from the fact that they are associated to several physical phenomena of interest, such as diffusion, social networks \cite{jiang2014diffusion}, and neural networks \cite{wang2013synchronization}. Recently, this class of systems has been used to model diffusion of epidemics in communities \cite{berestycki2021propagation}, therefore it is natural to consider how this class of PDEs behave in networks.

In \cite{6889559}, the authors tackled the problem of synchronization of a class of boundary controlled parabolic PDEs in which coupling between the agents occurs not only on the boundary but also in the domain. In this paper, we consider a similar setting of interconnected systems. However, we restrict the control to a subset of the agents only so that the synchronization of the network will occur for the controlled agents connected to the leader via the boundary control law and
for the other agents via the \textit{in-domain} couplings. 

 With respect to the current literature, the contribution of this paper is threefold: i) we consider a novel class of linear interconnected dynamical systems with both boundary and \textit{in-domain} couplings; ii) sufficient conditions in the form of matrix inequalities that ensures the synchronization of the network are provided; iii) the feasibility of the proposed matrix inequalities is thoroughly studied and sufficient conditions on the communication graph ensuring synchronization are established. The latter point in particular constitutes a contribution with respect to the existing literature as it allows to determine weather or not a certain network can achieve synchronization and how to find the appropriate coupling to achieve synchronization.

The remainder of the paper is organized as follows: in Section~\ref{sec:ProblemStatement} we introduce some preliminaries and the abstract problem formulation, while in Section~\ref{sec:Parabolic} we  formalize the problem in the case of nodes of the network with partially controlled parabolic dynamics. In Section~\ref{sec:SuffCond} we present the necessary and sufficient conditions on the control parameters and the communication graph to achieve synchronization with respect to the leader. A numerical example is given is Section~\ref{sec:Example}. We conclude with some final remarks in Section~VI.
\subsection{Preliminaries}
\subsubsection{Notation}
 $\mathcal{M}_N(\mathbb{R})$ denotes the set of square $N\times N$ real matrices, $\mathds{1}_N\in\mathbb{R}^N$ is the all-ones vector, and given a matrix $M$, $\Vert M \Vert_F$ indicates the Frobenius  norm of $M$.
Let $X$ be a normed linear vector space,
the symbol $I_X$ ($I_N$) is the identity operator in $X$ (matrix in $ \mathcal{M}_N(\mathbb{R})$). 
Let $a, b$ be real numbers, $\mathbf{L}^2(a,b;\mathbb{R}^n)$ denotes the quotient space of the space of Lebesgue measurable square integrable functions on $(a,b)$ with values in $\mathbb{R}^n$ with respect to the Lebesgue measure. The shorthand notation $\mathbf{L}^2(a,b;\mathbb{R})=\mathbf{L}^2(a,b)$ is used. The symbol $\mathbf{H}^n(a,b)$ stands for the set of $f\in \mathbf{L}^2(a,b)$ such that for all $i=1,2, \dots,n$, $f^i\in \mathbf{L}^2(a,b)$; where $f^i$ stands for the weak derivative of order $i$ of $f$. 
The symbol $D(A)$ stands for the domain of the operator $A$. Let $X$ be a real Hilbert space and $A\colon D(A)\subset X\rightarrow X$ be a linear operator, the notation $A\preceq 0$ indicates that for all $x\in D(A)$, $\langle x, A x\rangle\leq 0$.
For a symmetric matrix $M$, $M\succ 0$ and $M\prec 0$ denote, respectively, positive and negative definiteness. Given $a_1, a_2,\dots, a_n$, the symbol $\operatorname{diag}(a_1, a_2,\dots, a_n)$
stands for the diagonal matrix having $a_1, a_2,\dots, a_n$ as diagonal elements. The Kronecker (tensor) product $\otimes$ is used in the sense of \cite[Definition 4]{ferrante2021synchronization}. The symbol $\ker A$ stands for the kernel of the linear operator $A$.
\subsubsection{Graph theory}
A communication graph is described by an ordered pair ${\mathcal G}=\{{\mathcal V}, {\mathcal E}\}$ in which ${\mathcal V}$ is a set of $n$ {\em nodes} ${\mathcal V}=\{v_1,v_2, \ldots, v_n$\}, ${\mathcal E} \subset {\mathcal V}\times {\mathcal V}$ is a set of {\em edges} $\varepsilon_{jk}$ that models the interconnection between two nodes with the flow of information from node $j$ to node $k$. We denote by $L \in \mathbb{R}^{N \times N}$ the {\em Laplacian matrix} of the graph, with elements defined as $\ell_{kj} = - 1$ if there is an edge between node $k$ and node $j$ and $0$ otherwise for $k\ne j$, and $ \ell_{kk} = -\sum_{i=1,i\neq k}^N \ell_{ki}$. 
+
A path in $\mathcal{G}$ is a sequence of alternating vertices and edges $W=v_0\varepsilon_{01}v_1 \varepsilon_{12}v_2....\varepsilon_{n-1,n}v_n$ such that $\varepsilon_{ij}$ is an edge between $v_{i}$ and $v_{j}$. Two vertices $a$ and $b$ in $\mathcal{G}$ are called connected if there exist a path between $a$ and $b$. The graph is connected if the exist a path between every pair of vertices in $\mathcal{G}$ (see~\cite{godsil2001algebraic}).

%\subsubsection{Finsler's Lemma}
%\begin{lem}[Finsler's lemma]
%Let $L$ be a positive semi definite matrix, $L=A^T A$, and $Q\in \mathbb{R}^{n\times n}$ then the following statements are equivalent:
%\begin{itemize}
%\item $x^TQx <0$ for all $x\in \text{ker}(A)$
%\item There exist $\mu \in \mathbb{R}$ such that $Q-\mu L<0$
%\end{itemize}
%\end{lem}

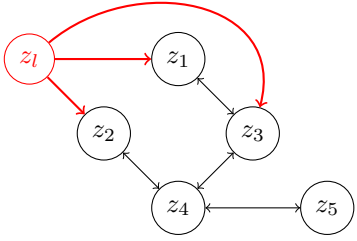
\begin{figure}[t!]
\begin{center}
 \begin{tikzpicture}[node distance={14mm}, main/.style = {draw, circle}] 
\node[main] (2) {$z_2$}; 
\node[main] (1) [above right of=2] {$z_1$};
\node[main] (3) [below right of=1] {$z_3$};
\node[main] (4) [below left of=3] {$z_4$};
\node[main] (5) [below right of=3] {$z_5$};
\node[main] (6) [red] [above left of=2 ] {$z_l$}; 
\draw [<->] (1) -- (3);
\draw [<->] (2) -- (4);
\draw [<->] (3) -- (4);
\draw [<->] (4) -- (5);
\draw [->] (6)[red, thick]  to  (1);
\draw [->] (6)[red, thick] to  (2); 
\draw [->] (6)[red, thick] to  [out=40,in=75,looseness=1.3] (3);  
\end{tikzpicture}
\caption{\label{fig:graph} Example of network considered in this paper: black connections represent the \textit{in-domain} connection between the systems while red connections represents the systems who are communicating with the leader.}
\end{center}
\end{figure}

\section{Problem statement}
\label{sec:ProblemStatement}

We consider a undirected network of $N$ systems together with a leader, the latter being labeled by the index $N+1$. The graph $\mathcal{G}$ associated to such a network can be separated $\mathcal{G}= \mathcal{G}_{l} \, \cup \, \mathcal{G}_{in}$ ($\mathcal{G}_{l} \cap \mathcal{G}_{in} = \emptyset$) where $\mathcal{G}_{l}$ describes the connection between the leader and the followers and $\mathcal{G}_{in}$ describes the interconnection among the followers (see Figure~\ref{fig:graph}). We assume that $\mathcal{G}$ is leader-connected.

\begin{figure}
\centering
\graphicspath{ {img/} }
\includegraphics[scale=0.3]{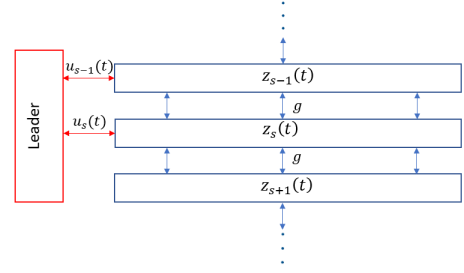}
\caption{ \label{fig:control} Visual representation of the network, with $s$ systems connected to the leader and the others connected through the \textit{in-domain coupling}}
\end{figure}

The $N$ followers are described by abstract dynamical systems of the form 
\begin{subequations} \label{2}
\begin{align}
    &\dot{z}_i= \mathfrak{A} z_i + g_i \sum_{j=1}^N  l_{ij}z_j+ f\label{1}\\
    &\mathfrak{B}z_i=u_i
\end{align}
\end{subequations}
for $i=1,...,N$ , with $l_{ij}$ elements of the Laplacian matrix~$L\in\mathbb{R}^{N\times N}$ associated to the graph $\mathcal{G}_{in}$ which encodes the network interconnections, and where
\begin{flalign*}
    &\mathfrak{A}: D(\mathfrak{A}) \hspace{0.5 em} \longrightarrow \hspace{0.5 em} \mathcal{X}, \,\,\, \mathfrak{B}: D(\mathfrak{B})\hspace{0.5 em} \longrightarrow \hspace{0.5 em} \mathcal{U}\\
    &D(\mathfrak{A}) \subset D(\mathfrak{B}) ,\hspace{0.5 em} \text{for all $i=1,...,N$}\\
    &f: \mathbb{R}^{+} \hspace{0.5 em} \longrightarrow \hspace{0.5 em} \mathcal{X}
\end{flalign*}
the state space $\mathcal{X}$ is a separable real Hilbert space. We suppose that $\mathcal{U}$ is a real vector space. The operator $\mathfrak{A}$ can be thought as the differential operator that governs the dynamics of the agents, and the term $ \sum_{j=1}^N  l_{ij}z_j(t)$  represent some \textit{in-domain couplings} among the different agents. The scalar $g_i \in \mathbb{R}$ is a scaling gain to be designed. Equation (\ref{2}b) is the boundary condition term, where $\mathfrak{B}$ is a trace operator, and $u_i$ in the input for each agent. The term $f$ is an external source term. Furthermore, we define the leader system as
\begin{subequations}
\label{leader}
\begin{equation}
\dot{z}_l= \mathfrak{A}_l z_l+ f
\end{equation}
with 
\begin{equation}
\begin{aligned}
&\mathfrak{A}_l=\mathfrak{A},  
&D(\mathfrak{A}_l)=D(\mathfrak{A})\cap\ker\mathfrak{B}    
\end{aligned}
\end{equation}
\end{subequations}
and accordingly $N$ error coordinate $e_i=z_i-z_l$, for $i=1,\ldots,N$, which represents the synchronization error with respect to the leader.
Bearing in mind that 
$\sum_{j=1}^N l_{ij}=0$, 
the error dynamics can be written as
\begin{subequations}
\begin{align}
    &\dot{e}_i= \mathfrak{A} e_i + g_i \sum_{j=1}^N  l_{ij}e_j \label{3} \\
    &\mathfrak{B}e_i=u_i, \hspace{0.4 em}    \label{4}
\end{align}
\end{subequations}

We define $z=\left(z_1, \cdots , z_N\right)$, $e=(e_1, \cdots , e_N)$, and $u=(u_1, \cdots , u_N)$, and we assume that the control input is selected as follows
\begin{align}
    u=\mathcal{C} (M \otimes I_{\mathcal{X}}) e \label{5}
\end{align}
where \[\displaystyle\mathcal{C}\colon D(\mathcal{C}) \subset \bigoplus_{i=1}^N \mathcal{X} \longrightarrow \bigoplus_{i=1}^N \mathcal{U}\] is a linear diagonal operator and $M \in \mathbb{R}^{N \times N}$ is a diagonal matrix associated to the subgraph $\mathcal{G}_l$ whose entries are $m_i=1$ if node $z_i$ is connected to the leader $z_l$ and $m_i=0$ otherwise.
% , i.e., $M$ is defined as
% \begin{equation}\label{eq:M}
% M=\left\lbrace\begin{array}{rcl}
%      m_i=1 &  \qquad &{\rm if} \; \exists \; \varepsilon_{i, N+1} \in \mathcal{E}_l \\
%      m_i=0 & \qquad & {\rm otherwise}
% \end{array} \right.    
% \end{equation}
% where $\mathcal{E}_l$ is the set of edges of $\mathcal{G}_l$. 
Then, the error dynamics can be written in a more compact form as follows:
\begin{subequations}
\label{eq:error_dyn}
\begin{equation}
    \dot{e}= \mathcal{A}e + \mathcal{L}_\mathcal{G}e \label{6}
\end{equation}
with:
\begin{align}
& \mathcal{A}= I_N \otimes \mathfrak{A},&&\\ 
& \displaystyle D(\mathcal{A})=\bigoplus_{i=1}^N D(\mathfrak{A}) \cap \ker ((I_N \otimes \mathfrak{B}) - \mathcal{C} (M \otimes I_{\mathcal{X}})) \hspace{0.5 em} \label{D(A)}  \\
\label{eq:defG}
& \mathcal{L}_\mathcal{G}=GL\otimes I_\mathcal{X} ,\hspace{0.5 em}  G=\operatorname{diag}(g_1,...,g_N)&&
\end{align}
\end{subequations}
We consider the following standing assumption.

\begin{assumption}
\label{assu:WellPosed}
The operators $\mathfrak{A}_l$ in \eqref{leader} and $\mathcal{A}$ in \eqref{eq:error_dyn} generate a strongly continuous semigroup, respectively, on $\mathcal{X}$ and $\mathcal{X}^N$.\hfill$\triangle$
\end{assumption}

Now observe that in the coordinates $(z_l, e)$ the set wherein synchronization occur reads:
\begin{equation}\label{set:sync}
\mathcal{S}_l= \{(z_l, e)\in \mathcal{X}^{N+1} \colon e=0\}
\end{equation}
Thus, the problem of synchronizing equations (\ref{2}) boils down to rendering the origin of the error dynamics \eqref{eq:error_dyn} globally exponentially stable.
Recall now that since the operator $\mathfrak{A}$ generates a strongly continuous semigroup, then also $\mathcal{A}$ generates a strongly continuous semigroup. Furthermore, $\displaystyle \forall z\in \mathcal{X}^N$, we have $\Vert\mathcal{L}_\mathcal{G}z\Vert_{\mathcal{X}^N}\leq \Vert GL\Vert_{F} \hspace{0.2em}\Vert z\Vert_{\mathcal{X}^N} $ so $\mathcal{L}_\mathcal{G}$ is a bounded linear operator, then by the perturbation theorem \cite[Theorem 3.2.1]{curtain1995linear}, we have that $\mathcal{A}+\mathcal{L}_\mathcal{G}$ also generates a strongly continuous semigroup.

Next, we provide sufficient conditions to ensure exponential synchronization of the family of systems (\ref{2})-\eqref{leader} interconnected via \eqref{5}. The following result, which is a straightforward adaption of \cite[Theorem 5.1.3, page 217]{curtain1995linear}.
\begin{proposition}\label{th:general_stability} 
Let Assumption~\ref{assu:WellPosed} hold.
Suppose that there exist a bounded positive operator $\mathcal{P}\colon\mathcal{X}^N\to\mathcal{X}^N$ and a positive real number $\delta$ such that $\mathcal{P(A}+\mathcal{L}_\mathcal{G}+\frac{\delta}{2} \mathcal{I})\preceq 0$. Then, the origin of \eqref{eq:error_dyn} is globally exponentially stable. This in turn implies that the set $\mathcal{S}_l$ defined in \eqref{set:sync} is globally exponentially stable  for  \eqref{2}-\eqref{leader} coupled via \eqref{5}. \hfill$\diamond$
\end{proposition}
% \begin{proof}
% Let for all $e\in \mathcal{X}^{N}$, 
% $V(e)=\langle \mathcal{P}e, e\rangle$, then $\dot{V}(e)\coloneqq \langle DV(e), \dot{e}\rangle= 2\langle \mathcal{P}\dot{e},e \rangle = 2 \langle \mathcal{P(A+}\mathcal{L}_\mathcal{G})e,e \rangle$. Now $\dot{V}(e)+\delta V(e)=2 \langle \mathcal{P(A+}\mathcal{L}_\mathcal{G}+\frac{\delta}{2} \mathcal{I})e,e \rangle $. Since $\mathcal{P(A+}\mathcal{L}_\mathcal{G}+\frac{\delta}{2} \mathcal{I})\preceq 0$, for all $e\in\mathcal{X}$ we have $\dot{V}(e)+\delta V(e)\leq 0$. Thus, direct application of the Grönwal's inequality implies that $V$ approaches zero exponentially fast along the solutions to \eqref{eq:error_dyn}. Hence, by recalling that $\mathcal{P}$ is bounded and coercive, it follows that the set $\mathcal{S}_l$ is globally exponentially stable for \eqref{eq:error_dyn}. By definition, $e=0$ implies $z_i-z_l=0$ for all $i=1,\ldots,N$ which thus proves that all the system \eqref{2} synchronizes with the leader \eqref{leader}.
% \end{proof}}

\section{Partially controlled parabolic systems}
\label{sec:Parabolic}
In this section, we specialize the setting considered in the previous section to the case of partially controlled parabolic systems. In particular, we assume that $\mathcal{X}=\mathbf{L}^2(0,1)$, endowed with its standard norm, and that the data in \eqref{2} is as follows:
\begin{subequations}\label{eq:LaplacianOperator}
\begin{align}
    &\mathfrak{A}: D(\mathfrak{A})=\left\{z \in \mathbf{H^2} (0,1) \colon \frac{dz}{dx}(1) = 0\right\} \hspace{0.5 em} \longrightarrow \hspace{0.5 em} \mathbf{L}^2(0,1)\\
    &\mathfrak{A}z=\beta \frac{d^2 z}{d x^2} + \alpha z , \hspace{1em} \beta>0 , \hspace{1em} \alpha \in \mathbb{R} \\
    &\mathfrak{B}:  \mathbf{H}^2(0,1) \hspace{0.5 em} \longrightarrow \mathbb{R}\hspace{0.5 em}\\ 
    &\mathfrak{B}z= \frac{d}{dx}z(0)
\end{align}
\end{subequations}
\begin{remark}
Specializing the setup in Section~\ref{sec:ProblemStatement} to the considered class of parabolic systems enables to come up with a set of sufficient conditions for synchronization that can be easily checked. This is the objective of the remainder of this paper.     
\end{remark}
%$g_i$ and $l_{ij}$ are defined as above and $f$ is an in-domain source generation term. In particular, in this setting $\mathcal{X}=L^2(0,1)$ with its usual inner product. 

 To define the control inputs, we split the agents into two sub-groups, i.e., the \textit{leader-disconnected} which do not have access to the leader and the \textit{leader-connected} agents that can exchange information directly with the leader. Without loss of generality, we label the latter from $i=1,\cdots, s\leq N$ and we define the local control input $u_i$ as
\begin{align}
    \displaystyle u_i=\int_0^1 k_{i} m_{i} e_i (x)dx. \label{8}
\end{align}
with $m_i$ element of the matrix $M$ introduced in \eqref{5} and $k_{i} \in \mathbb{R}$ are the controller gain to be designed. 
It is worth noticing that the protocol \eqref{8} is distributed in the sense that only the local error $e_i=z_i-z_l$ is available. The relative errors $e_i-e_j=z_i-z_j$, which contribute to the classic diffusive coupling, are the terms that drive the \textit{in-domain} couplings in \eqref{1}.

% {\color{red} Originally the inputs were of the form $\displaystyle u_i=\int_0^1 \sum_{j=1}^s k_{ij} e_j(x)dx.$ I don't recall why we did the change, either for simplicity (even though we get basically the same result on the LMI) or simply because \eqref{8} works anyways.}
% {\color{green}
% Agent $i$ cannot access $e_j$ otherwise the protocol would not be distributed. After (4) I have indeed added that $\mathcal{C}$ should be a diagonal operator to satisfy the distributed constraints.
% }

The agents not communicating with the leader can only exchange information with other agents via the \textit{in-domain coupling}, in other words for $i=s+1 ,\cdots ,N$, we have $u_i=0$. Finally, we will say that the network is leader-to-all connected or fully controlled if $s=N$ and say partially controlled otherwise. With this choice in mind, the operator $\mathcal{C}$ in \eqref{5} specializes into:
\begin{equation}
\begin{aligned}
&D(\mathcal{C})=\mathbf{L}^2(0,1;\mathbb{R}^N)
&\mathcal{C}h= \left(K \otimes \int_0^1\right) h
\end{aligned}
\label{eq:CalIntegral}
\end{equation}
where $K\coloneqq\operatorname{diag}\lbrace k_i \rbrace_{i=1,\ldots,N}$. Observe that the selection of the data in \eqref{eq:LaplacianOperator}-\eqref{eq:CalIntegral} ensures that Assumption~\ref{assu:WellPosed} holds for the specific class of systems considered henceforth. Indeed, from  \eqref{eq:LaplacianOperator}, $\mathcal{A}_l$ as defined in \eqref{leader} turns out to be the operator associated to the \emph{heat equation with Neumann boundary conditions}, which generates a strongly continuous semigroup on the space $\mathbf{L}^2(0, 1)$; see \cite[Example 2.3.7]{curtain1995linear}. Moreover, a standard eigenvalue analysis coupled with \cite[Theorem 2.3.5, item c, page 41]{curtain1995linear} enables to show that $\mathcal{A}$ generates a strongly continuous semigroup as well; see also Remark~\ref{rem:EigenValues}.

A visual representation of the control architecture is shown in Figure \ref{fig:control}, with $s \leq N$ nodes connected to the leader and the others coupled with \textit{in-domain} connection. We are now ready to introduce the main result of this section. 
\begin{thm}
Let
\begin{align}
&\bar{K}=KM\\
&D=-\beta P\bar{K}+\alpha P + PGL  \label{eq:D_omega}
\end{align}
Suppose that there exists a positive definite matrix $P\in\mathcal{M}_N(\mathbb{R})$ such that
\begin{equation}
    \Omega\coloneqq\begin{bmatrix}
 -\frac{\beta \pi^2}{2} P& \beta P\bar{K}\\\beta (P\bar{K})^T & D+D^T
\end{bmatrix}\prec 0 \label{9}
\end{equation}
Then system \eqref{2}, \eqref{eq:LaplacianOperator} with inputs \eqref{8} achieve synchronization towards the leader with respect to the $\mathbf{L}^2$-norm.
\end{thm}
\begin{proof}
Let $\mathcal{P}=P\otimes I_\mathcal{X}$ and observe that $\mathcal{P}$ is a positive operator. Then, one gets: 
\begin{equation}
\begin{aligned}
&2\langle\mathcal{P(A+\mathcal{L}_\mathcal{G}})e,e\rangle=2\beta\int_0^1 e^T P\frac{d^2}{dx^2}e+2\alpha \int_0^1 e^T Pe\\
&+2\int_0^1 e^T PGLe=2\beta e^T P\frac{d}{dx}e\Big\vert_{0}^{1}-2\beta\int_0^1 \frac{d}{dx}e^T P\frac{d}{dx}e\\
&+2\alpha \int_0^1 e^T Pe +2\int_0^1 e^T PGL e \\
&=-2 \beta e^T(0)P\bar{K}\int_0^1e -2\beta \int_0^1 \frac{d}{dx}e^T P\frac{d}{dx}e \\
&+2\alpha \int_0^1 e^T Pe + 2\int_0^1 e^T PGL e 
\end{aligned}
\label{eq:calc_int}
\end{equation}
where, for simplicity, we dropped the independent variable.
Thus, by denoting $\hat{e}=e-e(0)$, the following holds:
\begin{align}
&e^T(0)P\bar{K}\int_0^1e=
-\int_0^1(\hat{e}^T+e) P\bar{K}e \label{eq:integral1}
\end{align}
Moreover, by noticing that
$$
-\int_0^1 \frac{d}{dx}e^T P\frac{d}{dx}e=-\int_0^1 \frac{d}{dx}\hat{e}^T P\frac{d}{dx}\hat{e}
$$
and by using the so-called variation of Wirtinger’s inequality; see \cite[Page 17]{krstic2008boundary}, the latter gives:
%&\int_0^(e^T(0,t)-e^T(x,t)+e^T(x,t))P\bar{K}e(x,t)dx\\=
\begin{align}
&-\int_0^1 \frac{d}{dx}e^T P\frac{d}{dx}e\leq -\frac{\pi^2}{4}\int_0^1 \hat{e}^T P\hat{e} \label{eq:integral2}
\end{align}
Combining \eqref{eq:calc_int} with \eqref{eq:integral1} and \eqref{eq:integral2}, and by defining $\tilde{e}=(\hat{e},e)$, we get
\begin{equation}
\label{eq:OmegaBound}
\begin{aligned}
2\langle\mathcal{P(A+\mathcal{L}_\mathcal{G}})e,e\rangle \leq \int_0^1 \Tilde{e}^T \Omega \tilde{e}
\end{aligned}
\end{equation}
with $\Omega$ and $D$ defined, respectively, in \eqref{eq:D_omega} and \eqref{9}. From \eqref{9}, there exists $\delta>0$ such that $\Omega+\delta I_{2N}\preceq 0$. Moreover, notice that since  $\displaystyle\delta\langle e, e\rangle \leq \delta\langle \tilde{e}, \tilde{e}\rangle$, from \eqref{eq:OmegaBound} one gets:
\begin{flalign*}
2\left\langle\mathcal{P}\left(A+\mathcal{L}_\mathcal{G}+\frac{\delta}{\text{$2$}}. \!\mathcal{I}\right)e,e\right\rangle &\leq \int_0^1\!\! \tilde{e}^T\left(\Omega+\delta I_{2N}\right)\tilde{e}
\leq 0
\end{flalign*}
Thus, by invoking Proposition~\ref{th:general_stability} the results is established.
\end{proof}

\section{Sufficient conditions for synchronization}
\label{sec:SuffCond}
In the previous section, sufficient conditions in the form of matrix inequalities for synchronization of a class of parabolic interconnected PDEs are given. In this section, we  analyze the effect of the control parameters on the synchronization dynamics both in the fully controlled ($s=N$) and partially controlled ($s<N$) case. To do so, in the remainder of the paper, we consider the following simplifying assumptions:
\begin{equation}
 \beta=1, P=I_N, G=gI_N, K=kI_s   
\label{eq:Params}
\end{equation} 
 which in particular imply that all the control gains $k_i=k$ and all the \textit{in-domain} scaling $g_i=g$ are identical. The fact of having a common gain for all agents is ubiquitous in networks control (see for instance \cite{dorfler2013synchronization}, \cite{Scardovi2008SynchronizationIN} and \cite{4287131}). Furthermore, note that, under \eqref{eq:Params}, the matrix inequalities \eqref{9} becomes an LMI.

\subsection{ Fully controlled case}
First, we consider the fully controlled  scenario, i.e., $s=N$ namely all the nodes communicate with the leader. The following result holds.

\begin{lem}
Consider the network of \eqref{2}, with \eqref{eq:LaplacianOperator}, coupled with the leader \eqref{leader} through \eqref{8}. Let $\alpha$ and $k$ be such that:
\begin{equation}\label{eq:omega_full}
\Omega_{N1}:=\begin{bmatrix}
 -\frac{\pi^2}{2} I_N& k I_N\\
k I_N & 2(\alpha - k)I_N
\end{bmatrix}\prec 0
\end{equation}
Then, for any $g \leq 0$ inequality $\eqref{9}$ holds and synchronization is achieved with respect to the $\mathbf{L}^2$-norm.
\end{lem}

\begin{proof}
%Suppose that $L\neq0$, namely there is at least a connection between the systems in the network. Then, condition \eqref{9} turns into
In light of \eqref{eq:Params}, \eqref{9} reads as
\begin{align}\label{proof:lem1}
\Omega=\Omega_{N2}:=\begin{bmatrix}
 -\frac{\pi^2}{2} I_N& k I_N\\
k I_N & 2(\alpha - k)I_N + gL
\end{bmatrix}\prec 0
\end{align}
Note that \eqref{proof:lem1} can be written as $\Omega_{N2}=\Omega_{N1}+L_s$ where $L_s=\begin{bmatrix}
0 & 0\\0 & gL
\end{bmatrix} \preceq 0$. Then if \eqref{eq:omega_full} holds, this guarantees that \eqref{proof:lem1} holds too and thus in view of Theorem \ref{th:general_stability}, synchronization is achieved. 
\end{proof}

This last result proves that in the fully controlled case, the \textit{in-domain coupling} term $gL$ plays no \textit{necessary} role in achieving synchronization. Therefore, we shift our attention to $\Omega_{N1}$ in \eqref{eq:omega_full}. Notice that $\Omega_{N1}=\Tilde{\Omega} \otimes I_N$ where
$$
\Tilde{\Omega}=\begin{bmatrix}
 -\frac{\pi^2}{2} & k \\
k  & 2(\alpha - k)
\end{bmatrix}$$ 

Thus, \eqref{eq:omega_full} holds if $\Tilde{\Omega}$ is definite negative. Negative definiteness of $\Tilde{\Omega}$ is equivalent to the conditions $\operatorname{trace}(\Tilde{\Omega})<0$ and $\det(\Tilde{\Omega})>0$, which lead to
\begin{flalign}
    &k>\alpha-\frac{\pi^2}{4} \label{10}\\
    &k^2-\pi^2k+\pi^2\alpha<0 \label{11}
\end{flalign}
In particular, solving \eqref{11} with respect to $k$, it turns out that \eqref{11} is equivalent to
\begin{align}
 \label{13}
    &\frac{\pi^2}{2}-\frac{\pi}{2}\sqrt{\pi^2-4\alpha}<k<\frac{\pi^2}{2}+\frac{\pi}{2}\sqrt{\pi^2-4\alpha}\\
    &\frac{\pi^2}{4}-\alpha>0  \label{12}
\end{align}
Thus, given $\alpha$ satisfying (\ref{12}), $k$ needs to be selected so that \eqref{10} and \eqref{13} hold.

\begin{remark}
\label{rem:EigenValues}
To fully understand the effect of the control action, consider the open-loop equations, that is $k=0$. Then, the only possibility for the systems to synchronize on the leader for different initial conditions is to to have the source term $f=0$. Then, solutions to the error dynamics \eqref{6} converge to zero if and only if the operator $\mathfrak{A}$ defined in \eqref{eq:LaplacianOperator} and with domain redefined, according to $k=0$, as
\begin{equation}
    \label{eq:newDom}
\displaystyle D(\mathfrak{A})=\left\{z \in \mathbf{H}^2(0,1)\colon \frac{dz}{dx}(0)=\frac{dz}{dx}(1)=0\right\}
\end{equation}
It is well-known that exponential stability of the associated strongly continuous semigroup holds if $\displaystyle\sup_{n\geq 1}\Re(\lambda_n)<0$, where $\{\lambda_n, n\geq 1\}$ are the eigenvalues of $\mathfrak{A}$ with \eqref{eq:newDom}; see
\cite[Theorem 2.3.5, items c and d]{curtain1995linear}. Standard computations show that
$\lambda_j=\alpha-j^2 \pi^2$ for $j=1, 2,\dots$. Therefore, global exponential stability of the error dynamics holds if an only if $\alpha<0$. On the other hand, condition \eqref{12} shows that when $0\leq \alpha<\frac{\pi^2}{4}$ synchronization is achieved for suitable selection of $k$. This shows that the proposed synchronization policy enables to achieve synchronization even when the local dynamics are unstable.  
\end{remark}

\subsection{Partially controlled case}

We consider now the case in which $s<N$, namely not all the agents communicate with the leader. Thus in order to achieve synchronization the \textit{in-domain coupling} will play a fundamental role. For the sake of simplicity, we introduce the following notation:
\begin{align}
&\bar{K}= k\underbrace{\begin{bmatrix}
 I_s& \mathbf{0}_{s\times (N-s)}\\
\mathbf{0}_{(N-s)\times s} & \mathbf{0}_{(N-s)\times (N-s)}
\end{bmatrix}}_{P_s}
\label{eq:Ps}
%\\&\Omega_{N1}=\Omega(s=N,L=0)\\
%&\Omega_{N2}=\Omega(s=N)\\
%&\Omega_{s}=\Omega(s\leq N)
\end{align}

\begin{thm}\label{th:main_res}
Consider the network of \eqref{2}, with \eqref{eq:LaplacianOperator}, coupled with the leader \eqref{leader} through \eqref{8}. Let the graph $\mathcal{G}=\mathcal{G}_{l} \, \cup \, \mathcal{G}_{in}$ be connected. Suppose that the following conditions holds:
\begin{flalign}
        \alpha &< \frac{s\pi^2}{4N} \label{15} \\ 
        \frac{\pi^2}{2}-\frac{\pi}{2} \sqrt{\pi^2-4\frac{N}{s}\alpha}&<k<\frac{\pi^2}{2}+\frac{\pi}{2} \sqrt{\pi^2-4\frac{N}{s}\alpha} \label{16}
\end{flalign}
Then, there exists a scalar $g<0 $ such that $\eqref{9}$ holds and thus synchronization is achieved with respect to the $\mathbf{L}^2$-norm.
\end{thm}

\begin{proof}
In light of \eqref{eq:Params} and $s<N$, \eqref{9} reads as
\begin{equation}
\label{eq:Omegas}
\Omega=\Omega_{s}:=\begin{bmatrix}
 -\frac{\pi^2}{2} I_N& k P_s\\
k P_s & 2\alpha I_N - 2kP_s +gL
\end{bmatrix}\prec 0
\end{equation}
Therefore, from Schur's complement the following equivalence can be established
\begin{align}
   -\Omega_s\succ 0\iff \mathcal{D} \coloneqq 2kP_s-2\alpha I_N-gL-2\frac{k^2}{\pi^2}P_s \succ 0 \label{17}
\end{align}
Thus, from \eqref{eq:Omegas}-\eqref{17}
\begin{equation}
    \mathds{1}_N^T \mathcal{D} \mathds{1}_N= 2ks-2\alpha N- 2\frac{k^2}{\pi^2}s>0 
\label{18}
\end{equation}
which, by solving with respect to $k\in\mathbb{R}$, leads to (\ref{15}) and (\ref{16}). 
Let $Q\coloneqq 2\alpha I_N-2kP_s +\frac{2k^2}{\pi^2}P_s$ and $L=U^TU $, with $U$ being the incidence matrix. Combining \eqref{17} and Finsler's lemma (see, e.g., \cite{boyd1994linear}), the following items turn out to be equivalent:
\begin{itemize}
\item[$(i)$] $\exists g \in \mathbb{R}$ : $\Omega_s\prec0$
\item[$(ii)$] $x^TQx <0$ for all $x\in \text{ker}(U)$
\end{itemize}
We conclude the proof by showing that item $(ii)$ follows from \eqref{18}. To this end, notice that since by assumption the graph is connected, one has that $\ker(U)=\operatorname{span}\{\mathds{1}_N\}$. At this stage, 
observe that $\mathds{1}_N^T Q\mathds{1}_N= -2ks+2\alpha N + 2\frac{k^2}{\pi^2}s$. Thus, from \eqref{18}, $\mathds{1}_N^T Q\mathds{1}_N<0$. This shows that item $(ii)$ above holds, thereby concluding the proof.
\end{proof}

\begin{remark}
    If the graph $\mathcal{G}$ is not connected, then we can apply Theorem \ref{th:main_res} for each of its $p$ connected component $\mathcal{G}_i$, $\cup_{i=1}^p \mathcal{G}_i=\mathcal{G}$, each with its own $N_i$ and $s_i$ that should satisfies equation \eqref{15}, \eqref{16}. The only case in which such $g$ does not exist is when there is an isolated connected component $\mathcal{G}_i$ whose nodes do not communicate with the leader (namely there is an $i$ such that $s_i=0$). 
\end{remark}

\begin{figure}
\graphicspath{ {img/} }
\psfrag{t}[1][1][1]{$t$}
\psfrag{ub}[1][1][1]{$$}
\psfrag{z1t}[1][1][.6]{\quad$z_1(1,t)$}
\psfrag{z2t}[1][1][.6]{\quad$z_2(1,t)$}
\psfrag{z3t}[1][1][.6]{\quad$z_3(1,t)$}
\psfrag{z4t}[1][1][.6]{\quad$z_4(1,t)$}
\psfrag{z5t}[1][1][.6]{\quad$z_5(1,t)$}
\psfrag{zlt}[1][1][.6]{\quad$z_l(1,t)$}
\includegraphics[width=9cm ,height=4.7cm]{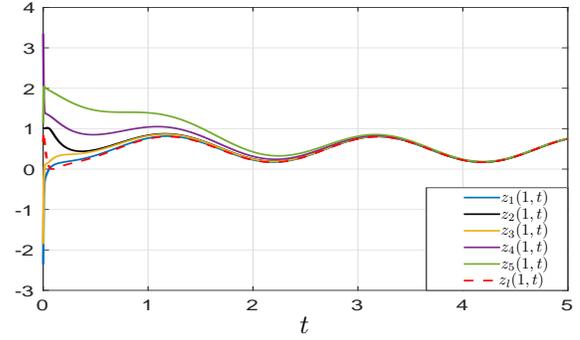}
\caption{\label{fig:error_time} Synchronization of the followers towards the leader: the boundary $z_i(1,t)$ evolution over time.}
\end{figure}

\begin{figure}
\graphicspath{ {img/} }
\psfrag{zs}[1][1][1]{$x$}
\psfrag{et01}[1][1][.7]{\qquad$\overline{e}(\cdot, 0.1)$}
\psfrag{et05}[1][1][.7]{\qquad$\overline{e}(\cdot, 0.5)$}

\psfrag{et1}[1][1][.7]{\qquad$\overline{e}(\cdot, 1)$}

\psfrag{et2.5}[1][1][.7]{\qquad$\overline{e}(\cdot, 2.5)$}
\includegraphics[width=9cm ,height=4.7cm]{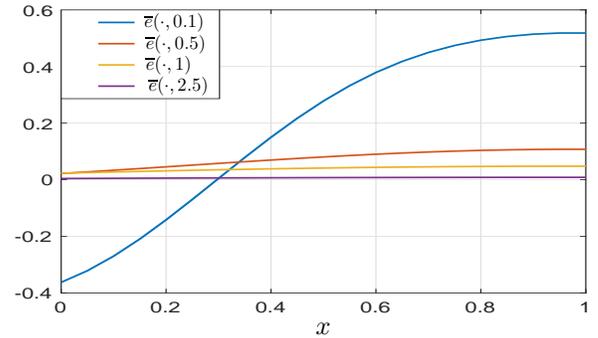}
\caption{\label{fig:error_av}Time evolution of $\bar{e}$ in \eqref{eq:av_error}. As we can clearly see, the mean error converges to $0$ with time.}
\end{figure}

\section{Numerical simulations}
\label{sec:Example}
We consider a group of $N=5$ agents of which $s~=~3$ are connected to the leader. The dynamics of each agent are governed by \eqref{2}-\eqref{eq:LaplacianOperator} with $\alpha=0$, and $ f(t)=(x\mapsto (1+\cos(2\pi x)) \sin(\pi t))$. The interconnection topology we consider is as in \figurename~\ref{fig:graph}. We design $k$ by using \eqref{16} and then seek for a $g$ that verifies \eqref{9}. In particular, since Theorem $2$ guarantees the existence of $g$, by fixing $k=3$ so that \eqref{16} is satisfied, one can easily solve \eqref{proof:lem1} (which is a linear matrix inequality) in $g$. Indeed, by selecting $g=-2$, the inequality \eqref{proof:lem1} is fulfilled.
\figurename~\ref{fig:error_time} shows the evolution of the boundary\footnote{Simulations have been performed in \texttt{Matlab} using the finite difference method.} $z_i(1,t)$ from the initial condition:
$z_{1,0}(x)=0.5+2\cos(5\pi x)+ \cos(\pi x)$, $z_{2,0}(x)=1, z_{3,0}(x)= 2\cos(5\pi x), z_{4,0}(x)= 1.5-2\cos(5\pi x), z_{5,0}(x)= 0.5 \cos (7\pi x)$ and $z_{l,0}(x)=2+\cos(\pi x)+2\cos(7 x)$.. The picture clearly confirms that the actual states synchronize. In \figurename~\ref{fig:error_av}, we show the evolution of the average error 
\begin{align}
 \label{eq:av_error} \overline{e}(x,t)=\sum_{i=1}^N e_i(x,t)
 \end{align} 
at different times. The figure suggests that synchronization happens 
In \figurename~\ref{fig:5} and \figurename~\ref{fig:6}, we consider respectively the case in which the connection to the leader are lost and the case in which the \textit{in-domain coupling} is absent. In the former, i.e., $k=0$, the agents still achieve synchronization among them but not on the leader's trajectory. In the latter, i.e., $g=0$, only the $s=3$ agents connected to the leader achieve synchronization on the leader. These results confirm the theoretical findings presented previously and show that in a general setting, both the boundary control and the \textit{in-domain coupling} are necessary to achieve synchronization of the full network.

%\begin{figure}
%\graphicspath{ {img/} }
%\includegraphics[width=9cm ,height=5cm]{img/domain_k=g=0.eps}
%\includegraphics[width=9cm ,height=5cm]{img/boundary_k=g=0.eps}
%\caption{The case where $k=g=0$. No coupling between any agents or with the leader. The trajectories are determined by the initial conditions and the dynamics. }
%\end{figure}

\begin{figure}
\graphicspath{ {img/} }
\psfrag{t}[1][1][1]{$t$}
\psfrag{ub}[1][1][1]{$$}
\psfrag{z1t}[1][1][.6]{\quad$z_1(1,t)$}
\psfrag{z2t}[1][1][.6]{\quad$z_2(1,t)$}
\psfrag{z3t}[1][1][.6]{\quad$z_3(1,t)$}
\psfrag{z4t}[1][1][.6]{\quad$z_4(1,t)$}
\psfrag{z5t}[1][1][.6]{\quad$z_5(1,t)$}
\psfrag{zlt}[1][1][.6]{\quad$z_l(1,t)$}
\includegraphics[width=9cm ,height=4.7cm]{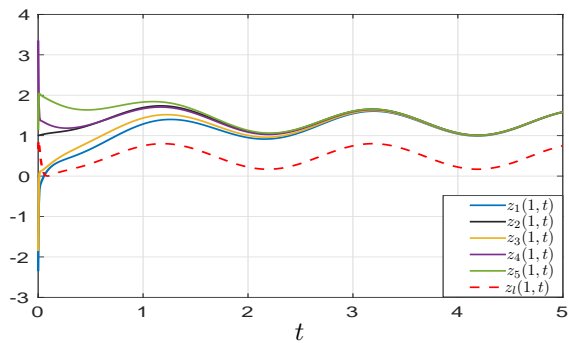}
\caption{The case where $k=0$. None of the agents is connected to the leader, however the agents can communicate with each other. We see the synchronisation behavior due to the dynamics but not towards the leader.}
\label{fig:5}
\end{figure}

\begin{figure}
\graphicspath{ {img/} }
\psfrag{t}[1][1][1]{$t$}
\psfrag{ub}[1][1][1]{$$}
\psfrag{z1t}[1][1][.6]{\quad$z_1(1,t)$}
\psfrag{z2t}[1][1][.6]{\quad$z_2(1,t)$}
\psfrag{z3t}[1][1][.6]{\quad$z_3(1,t)$}
\psfrag{z4t}[1][1][.6]{\quad$z_4(1,t)$}
\psfrag{z5t}[1][1][.6]{\quad$z_5(1,t)$}
\psfrag{zlt}[1][1][.6]{\quad$z_l(1,t)$}
\includegraphics[width=9cm ,height=4.7cm]{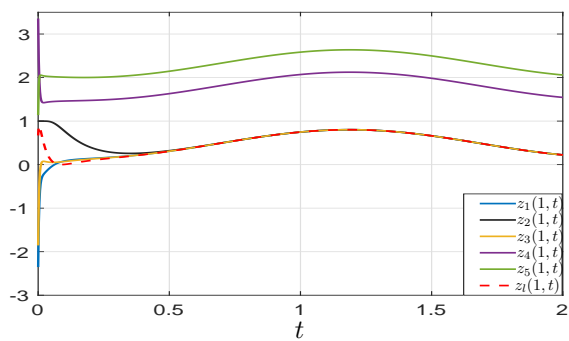}
\caption{The case where $g=0$. No interconnection between the agents, however $s=3$ agents are connected to the leader. As we can clearly see, only these agents synchronises to the leader while the other agents remain isolated.}
\label{fig:6}
\end{figure}

\section{Conclusion and future outlook}
 In this paper, we considered the problem of synchronization of a class of interconnected  infinite-dimensional dynamical systems. The problem is first presented in a general setup and later is specialized to partially controlled parabolic equations with interconnections among the agents taking place both at the boundary and {\it in-domain}. Sufficient conditions for synchronization towards the leader in the form of matrix inequalities have been established. A thorough analysis of the feasibility of such conditions is carried out. 
 The theoretical results are supported by numerical examples, where different network topologies are considered to illustrate the variety of possible synchronization behaviours for the agents.  Future studies will be focused on a deeper analysis of the role of the network topology in the synchronization process as well as nonlinear PDEs. Inspired by epidemics diffusion in communities, an application to cluster synchronization is part of our ongoing research. The extension towards more general dynamics, such as coupled ODE--PDEs, is also under study. 
\bibliography{biblio.bib}
\end{document}